\documentclass[12pt]{amsart}
\usepackage{ShaharStyle}

\title{Versal Family of Reductive Groups}
\author{Shahar Dagan}
\begin{document}

\begin{abstract}
    In this paper, we prove Theorem \ref{intthm: versal family reductive}, which establishes the existence of a family of reductive groups, including all reductive groups up to a given rank. In Theorem \ref{thm: Scheme of quasi-split}, we also demonstrate the construction of a similar versal family of quasi-split reductive groups. This result generalizes \cite[Lemma 3.2.1]{AA24} and provides a framework for systematically studying invariants of reductive groups with bounded rank.
\end{abstract}

\maketitle
\section{Introduction} \label{sec: introduction}
Reductive groups form an important class of algebraic groups with a well-established classification theory. One ingredient of this classification is the notion of  rank for a given reductive group. A fundamental tool in algebraic geometry is the construction of families of objects parameterized by a scheme. This is typically achieved by defining a morphism of schemes $X\to S$, where $S$ serves as the parameter space and the fibers of the morphism represent the family of objects. The main result of this paper is the construction of a family of all reductive groups up to a given rank, as described in Theorem \ref{intthm: versal family reductive},

\begin{introtheorem}[see Theorem {\ref{thm: versal family reductive}} below]\label{intthm: versal family reductive}
    For any integer $n>0$, there exists a scheme $\cS _n$ of finite type and a smooth  $\cS_n$-group scheme of finite type $\cR_n\to \cS_n$ such that:
    \begin{enumerate}
        \item For any field $F$ and every $x\in\cS_n(F)$, the group $(\cR_n)|_x$ is connected and reductive of rank at most $n$. In particular, $\cR_n$  is a reductive $\cS_n$-group scheme.
        \item For any connected and reductive group $\mathbf{G}$, of rank at most $n$ over a field $F$,  there exists a point $x\in \cS_n(F)$, such that $\mathbf{G}\cong (\cR_n)|_x$.
    \end{enumerate}
\end{introtheorem}

This result may prove useful in establishing uniform results for all reductive groups of bounded rank. If one finds some natural invariants for certain group schemes, computing such invariants on the scheme of Theorem \ref{intthm: versal family reductive} ensures that these invariants characterize all reductive groups up to a given rank.

This theorem generalizes \cite[Lemma 3.2.1]{AA24}, which established a similar result over finite fields. One may review \cite{AA24} to obtain a concrete example of how such families can be used.

\subsection{Structure of the proof}
Our approach builds on the method presented in \cite[Appendix B]{AA24}, where we encode the classification theory of reductive groups (as presented in \cite{Milne_reductive_groups,Con14} and originally shown in \cite{SGA3}) in the framework of scheme theory. Specifically, we encode all necessary data for classifying a reductive group into the base scheme $\cS_n$ and construct a reductive $\cS_n$-group scheme, $\cR_n$, following the method of classification of reductive groups, as we now describe.

The classification consists of two major steps: first, classifying all split reductive groups by their root data, and second, classifying the forms of split reductive groups using Galois cohomology. This approach is sufficient because any reductive group $\mathbf{G}$ splits over a finite Galois extension (see, e.g., \cite[17.16]{Milne_reductive_groups}). Moreover, it suffices to consider Galois extensions of degree bounded by a function depending only on the rank of $\mathbf{G}$ (see, e.g., \cite[Lemma B.1.1]{AA24}). The classification of forms also breaks into two steps, which we outline below.

The first step is to classify the quasi-split forms of a split reductive group $\mathbf{G}$ over a field $F$ (see the definition in \cite[15.4]{Milne_reductive_groups}). Since $\mathbf{G}$ is split, we can choose a pinning for it (see, e.g., \cite[Definition 1.5.4]{Con14}). This pinning allows us to classify the quasi-split forms of $\mathbf{G}$ that become isomorphic to $\mathbf{G}$ over a finite Galois extension $E/F$ by the Galois cohomology classes of $\Out(\mathbf{G}_E)$, where $\Out(\mathbf{G}_E)$ denotes the outer automorphism group of $\mathbf{G}_E$.

The 1-coboundary condition, together with modding by inner automorphisms, makes the Galois action  on $\Out(\mathbf{G}_E)$ trivial. Therefore, the Galois cohomology classes of $\Out(\mathbf{G}_E)$ is in bijection with the classes of homomorphisms from the Galois group $\Gal(E/F)$ to $\Out(\mathbf{G}_E)$, up to adjunction. Moreover, the pinning allows us to construct an isomorphism from $\Out(\mathbf{G}_E)$ to $\Aut(\fX,\Delta)$, the automorphism group of a based root datum of $\mathbf{G}$. We summarize this paragraph in the following equality
\begin{align*}
    H^1(\Gal(E/F), \Out(\mathbf{G}_E))
     & \cong \Mor(\Gal(E/F),  \Out(\mathbf{G}_E)) / \operatorname{ad}( \Out(\mathbf{G}_E)) \\
     & \cong \Mor(\Gal(E/F), \Aut(\fX, S)) / \operatorname{ad}(\Aut(\fX, S)).
\end{align*}
This equality also show that $H^1(\Gal(E/F), \Out(\mathbf{G}_E)) $ is finite (see, e.g., \cite[Lemma B.3.2]{AA24}).

The second step in the classification of $F$-forms of the split group $\mathbf{G}$, is to classify inner forms of quasi-split forms of $\mathbf{G}$. This is sufficient as any $F$-form has a (unique) quasi-split inner form (see e.g. \cite[19.57]{Milne_reductive_groups}).

Each inner forms of a quasi-split form $\mathbf{G}^{qs}$ of $\mathbf{G}$, has a corresponding class in the Galois cohomology $H^1(\Gal(E/F),\Inn(\mathbf{G}^{qs}_E))$, where $\Inn(\mathbf{G}^{qs}_E)$ is the inner automorphism group of $\mathbf{G}^{qs}_E$. Recall, that two different cohomology classes in $H^1(\Gal(E/F),\Inn(\mathbf{G}^{qs}_E))$ may correspond to the same inner form. Our construction does not avoid this redundancy and may therefore include multiple instances of the same inner form. This is one of the reasons our result is versal and not universal.

We have completed our review of the classification of reductive groups and now discuss how it is implemented in our proof. A naive approach would be to take a disjoint union of all reductive groups over all prime fields (i.e., $\mathbb{Q}$ or $\mathbb{F}_p$). This construction results in a scheme that is not of finite type, and is therefore unlikely to be useful. Instead, we perform the disjoint union only for ingredients of the classification that have a finite set of options and construct a scheme of finite type to handle the remaining ingredients, using the relative theory of reductive group scheme as presented in \cite{Con14}.

As we bound the rank of the reductive groups in our family, the root data, the Galois groups \footnote{Referring to the groups themselves rather than the fields in the corresponding extensions.}, and the cohomology set $H^1(\Gal(E/F), \Out(\mathbf{G}_E))$ have finitely many options. The finiteness of these components is established in \cite[Appendix B]{AA24}.

The remaining ingredients are the Galois cohomology $H^1(\Gal(E/F),\Inn(\mathbf{G}^{qs}_E))$ and the Galois extensions\footnote{This time referring to the fields.}. To encode the Galois extension schematically, we construct for any finite group $\Gamma$, a finite {\'e}tale morphism between finitely presented affine $\bZ$-schemes, $\Psi_{\Gamma} : \cE_{\Gamma} \rightarrow  \cF_{\Gamma}$, representing a family of finite {\'e}tale algebras of dimension $|\Gamma|$ endowed with a $\Gamma$-action. Moreover, for any field $F$ and an {\'e}tale $F$-algebra $E$ in the family described by $\Psi_\Gamma$, we get $E^\Gamma=F$ (see Lemma \ref{lem: fam of finite etale algebras}). These schemes include a point for any Galois extension over any field, such that the Galois group is $\Gamma$. The base scheme $\cF_\Gamma$ includes more points than needed for the classification process, which is another reason our result is only versal.

The definition of $H^1(\Gal(E/F), \Inn(\mathbf{G}^{qs}_E))$ involves a quotient of 1-cocycles by 1-coboundaries. Consequently, the associated functor\footnote{A functor on $E$-schemes mapping a scheme $S$ to $H^1(\Gal(E/F), \Inn(\mathbf{G}^{qs}_E(S)))$.} may not be representable by a scheme. Fortunately, to ensure versality, it suffices to construct a scheme representing 1-cocycles without quotienting by 1-coboundaries. This approach may result in redundant instances of the same form in the final family $\cR_n\to\cS_n$, which is one reason the result is versal rather than universal.

In summary, for any field $F$, an $F$-point in the base scheme $\cS_n$ encodes the following data:
\begin{enumerate}
    \item a based root datum $(\fX,\Delta)$;
    \item a finite group $\Gamma$;
    \item a finite {\'e}tale algebra $E$ over $F$ of rank $|\Gamma|$ with a $\Gamma$-action such that $E^\Gamma=F$;
    \item a homomorphism $\alpha:\Gamma\to\Aut(\fX,\Delta)$;
    \item a 1-cocycle in $Z^1(\Gamma,\Inn(\mathbf{G}^{qs})$), where $\mathbf{G}^{qs}_E$ is the quasi-split reductive group associated to $\fX$ and $\alpha$.
\end{enumerate}

The construction of the final scheme $\cR_n\to\cS_n$, strongly resembles the method of obtaining forms of a split reductive group, given the data described above.

\subsection{Structure of the paper}
The structure of the paper is as follows:
\begin{enumerate}
    \item  In this section we introduce the main notions of the paper.
    \item In Section \S\ref{sec: Preliminaries}, we give a brief review of the required concepts from scheme theory and the classification of reductive groups and reductive group schemes.
    \item  In Section \S\ref{sec: proof of main results}, we prove the main result of the paper. First, we construct a finite {\'e}tale scheme containing points describing finite Galois extensions. Then we construct a versal family of quasi-split reductive groups of a bounded rank. Finally, we prove Theorem \ref{intthm: versal family reductive} constructing a versal family of all reductive groups of a bounded rank.
\end{enumerate}

\subsection{Acknowledgement}
I would like to thank my advisor \textbf{Avraham Aizenbud} for guiding me in this project. I would also like to thank \textbf{Shachar Carmeli} for his valuable comments.

\section{Preliminaries} \label{sec: Preliminaries}
This paper assumes familiarity with scheme theory and the classification of reductive groups over fields. We align with the description of these theories as presented in \cite{Vak17} and \cite{SP} for schemes, and in \cite[Sections \S 1 and \S 7]{Con14} and \cite{Milne_reductive_groups} for reductive groups. Accordingly, in these preliminaries, we mainly review specific lemmas and constructions used in our proof.

We also rely on the relative theory, which extends the study of reductive groups to the setting of group schemes over a non-empty base scheme $S$. This perspective is reviewed in details in \cite{Con14}. In these preliminaries, we focus on the key definitions of the relative theory, in order to clarify the terminology used in our construction.

Informally, the definitions in the relative setting mostly align with the classical ones over any geometric point. Additionally, sub-objects such as Borel subgroups and tori are uniformly defined over the base scheme, in a manner that respects the structure of the reductive group scheme over the base scheme.

\subsection{Scheme theory constructions}
In this subsection, we fix a base scheme $S$ and an $S$-scheme $X$. We review several scheme-theoretic techniques for constructing new $S$-schemes from $X$. Most of these techniques are standard, and this review serves as a reminder and clarifies the terminology  and the notation used.

\begin{definition}
    The \textbf{functor of points} of $X$ is the contra-variant functor $\underline{X}:\mathbf{Sch}_S\to\mathbf{Set}$ defined as
    $$
        \underline{X}(T)\coloneqq\Mor_S(T,X),
    $$
    for any $S$-scheme $T$, and for a morphism $\phi:T\to T'$ of $S$-schemes by
    \[
        \underline{X}(\phi)(f)=f\circ\phi \in \underline{X}(T)
    \]
    for any $f\in \underline{X}(T')$. We refer to elements in $\underline{X}(T)$ as \textbf{$\mathbf{T}$-points} of $X$. If $T=\Spec(A)$ for some ring $A$, we also call elements in $\underline{X}(T)$ \textbf{$\mathbf{A}$-points} of $X$. We often omit the underline symbol and write $X(T)$ for $\underline{X}(T)$.
\end{definition}

Throughout this paper, we use the Yoneda lemma to show that two schemes are isomorphic, by showing that their functors of points are equivalent.

\begin{definition}
    A \textbf{geometric point} of the scheme $X$ is a point  $\overline{x}\in X(k)$, where $k$ is an algebraically closed field.
\end{definition}

\begin{definition}
    A contra-variant functor $F:\mathbf{Sch}_S\to\mathbf{Set}$ is said to be \textbf{representable} by an $S$-scheme $Y$ if $F$ is equivalent to $\underline{Y}$, the functor of points of $Y$. Similarly, a co-variant functor $F:\mathbf{ComRing}\to\mathbf{Set}$ is said to be \textbf{representable} by a scheme $Y$, if the composition of the functor $\Spec:\mathbf{ComRing}\to \mathbf{AffSch}\subseteq\mathbf{Sch}$  with $\underline{Y}:\mathbf{Sch}\to\mathbf{Set}$, the functor of points of $Y$,  is equivalent to $F$.
\end{definition}

In many of the upcoming claims, we define a functor by a notation such as $\underline{F}:\mathbf{Sch}_S\to\mathbf{Set}$, to imply it is representable, under certain assumptions, by an $S$-scheme $F$. Therefore, we are required to be careful and justify any subtraction of the underline symbol.

We now review several methods for constructing schemes in various settings.

\begin{lemma}\label{lem: finite lim of scheme exist}
    The category of (affine) $S$-schemes has all finite limit, i.e., for any finite diagram in $(\textbf{Aff})\textbf{Sch}_S$, the limit exists.
\end{lemma}
\begin{proof}
    This is a standard result of scheme theory, that uses the fact that any category with a terminal object and fibered products, has all finite limit. $(\textbf{Aff})\textbf{Sch}_S$ satisfies these conditions - it has all fibered products \cite[Theorem 9.1.1, Exercise 9.1.B.]{Vak17} and $S$ is a terminal object.
\end{proof}
\begin{lemma}\label{lem: limit base change}
    Finite limits of schemes commute with base change. More precisely, let $Y\coloneqq\underset{\to}{\lim}X_i$ be a finite limit of $S$-scheme and $S'$ be an $S$-scheme. Then \[
        Y\times_S S'\cong \underset{\to}{\lim} (X_i\times_S S')
    \]
\end{lemma}
\begin{proof}
    The isomorphism follows by the Yoneda lemma and a direct calculation on the functor of points.
\end{proof}

Now we describe how a group acts on a scheme, and how it allows us to construct new schemes.

\begin{definition}\label{def: group action on  a scheme}
    Let $\Gamma$ be a group. A \textbf{$\mathbf{\Gamma}$-action} on $X$ over $S$, is defined as a group homomorphism $\rho: \Gamma \to \Aut_S(X)$. In this case, the functor of points of $X$ refines to a functor $\underline{X}:\mathbf{Sch}_S\to\mathbf{Set}_\Gamma$.
\end{definition}

\begin{remark}\label{rem: functor to action on scheme}
    Notice, that if $\underline{X}:\mathbf{Sch}_S\to\mathbf{Set}_\Gamma$ composed with the forgetful functor $\mathbf{Set}_\Gamma\to\mathbf{Set}$ is representable by a scheme $X$, then by functoriality, we get a $\Gamma$-action on $X$.
\end{remark}

\begin{definition}\label{def: fixed points/quotient functor/space}
    Let $\Gamma$ be a  group, $\rho:\Gamma\to\Aut(X)$ a $\Gamma$-action on $X$ over $S$, and $T$ some $S$-scheme. Define the functor $\underline{X^{\rho(\Gamma)}}: \mathbf{Sch}_S \to \mathbf{Set}$ by
    $$\underline{X^{\rho(\Gamma)}}(T)=(\underline{X}(T))^{\rho(\Gamma)},$$
    which maps an $S$-scheme $T$ to the fixed points of the action on the functor of points. If this functor is representable, we denote the representing scheme by $X^{\rho(\Gamma)}$.

    Similarly, define the functor $\underline{X/\rho(\Gamma)}: \mathbf{Sch}_S \to \mathbf{Set}$ as the Zariski sheafification of the following presheaf in the Zariski topology
    $$\underline{X/\rho(\Gamma)}(T)=\{\text{orbits of the action of $\Gamma$ on $\underline{X}(T)$}\}.$$
    Similarly, if this functor is representable, we denote the representing scheme by $X/\rho(\Gamma)$.

    We sometimes omit $\rho$ from the notations if it is clear from context.
\end{definition}

We show in Lemma \ref{lem: fixed points scheme} that if the acting group is finite, then $\underline{X^{\rho(\Gamma)}}$ is always representable. However, the $\underline{X/\rho(\Gamma)}$ may not be representable. Nevertheless, in the case where the scheme is finite and {\'e}tale, and the group acting on it is finite, $\underline{X/\rho(\Gamma)}$ is representable (see e.g. \cite[Lemma 58.5.2]{SP}).

\begin{lemma}\label{lem: fixed points scheme}
    The $\underline{X^\Gamma}$, is representable and compatible with base change. More precisely, $\underline{X^\Gamma}$ is represented by an $S$-scheme $X^\Gamma$, and for any $S$-scheme $S'$ there is an equivalence of functors
    $$
        \underline{(X\times_S S')^\Gamma}\cong \underline{X^\Gamma\times_S S'}.
    $$
    Moreover, if $X$ is affine then so does $X^\Gamma$.
\end{lemma}
\begin{proof}
    Representability and preservation of affiness are immediate from \ref{lem: finite lim of scheme exist}. The compatibility with base change follows from \ref{lem: limit base change}.
\end{proof}

\begin{lemma}\label{lem: qoutient of etale}
    Let $\Gamma$ be a group, $\rho:\Gamma\to\Aut(X)$ a $\Gamma$-action on $X$ over $S$, and assume $X$ is finite and {\'e}tale over $S$. Then $\underline{X/\Gamma}$ is representable by a finite and {\'e}tale scheme $X/\Gamma$.

    Moreover, this construction is preserved under base change. More precisely, if $S'$ is an $S$ scheme, then the functor $\underline{(X\times_S S')/\Gamma}$ is representable by $(X/\Gamma)\times_S S'$.  In addition, if $X=\Spec(B)$ then $X/\Gamma\cong \Spec(B^\Gamma)$.
\end{lemma}
\begin{proof}
    Representability is proved in \cite[Lemma 58.5.2]{SP}. The isomorphism to $\Spec (B^\Gamma)$ follows from the anti-equivalence between affine schemes and commutative rings (see e.g. \cite[Exercise 6.3.D]{Vak17}). Therefore, we are only left to check compatibility with base change.

    Let $T$ be an $S'$-scheme, and $x'\in\underline{(X\times_S S')/\Gamma}(T)$. By definition, $x'$ corresponds to a $\Gamma$-orbit of $S'$-morphisms $\{x'_i:T\to X\times_S S'\}$. Now, as the $\Gamma$-action is over $S$, we get that $\{p_X\circ x'_i:T\to X\}$, where $p_X:X\times_S S'\to X$ is the projection, is also a $\Gamma$-orbit of morphism over $S$.  This corresponds to a morphism $x:T\to X/\Gamma$.  The morphism $x$, together with the $S'$-structure of $T$ induce an morphism $T\to(X/\Gamma)\times_S S'$.

    On the other hand, let $y:T\to (X/\Gamma)\times_S S'$ be an $S'$-morphism. Notice, that $p_{X/\Gamma}\circ y$ is a morphism from $T$ to $X/\Gamma$, where $p_{X/\Gamma}:(X/\Gamma)\times_S S'\to (X/\Gamma)$ is the projection. Therefore,  $p_{X/\Gamma}\circ y$ corresponds to a $\Gamma$-orbit of $S$-morphisms $\{y_i:T\to X\}$. Now, each $y_i$, together with the $S'$-structure of $T$ induce an $S'$-morphism $\{y_i':T\to X\times_S S'\}$. This set is a $\Gamma$-orbit by the induced $\Gamma$-action on $X\times_S S'$, hence an element in $\underline{(X\times_S S')/\Gamma}(T)$.

    We just defined a bijection, which is functorial in $T$, hence an equivalence of functors.
\end{proof}

In the next part, for any two $S$-schemes $X$ and $Y$, we define a functor from $\mathbf{Sch}_S\to\mathbf{Set}$ describing the $S$-morphism between $X$ and $Y$. Moreover, we present conditions under which it is representable, and the representing scheme is affine and of finite type.

\begin{definition}[see {\cite[Definition 3.4]{AA22}}]\label{def: internal hom}
    Let $Y$ be another $S$-scheme. Define the functor $\underline{X{}^{\wedge}_{S}Y}:\mathbf{Sch}_S\to \mathbf{Set}$ by
    $$\underline{X{}^{\wedge}_{S}Y}(T)=\Mor_S(Y\times_S T,X),$$
    for any $S$-scheme $T$.

    We call $\underline{X{}^{\wedge}_{S}Y}$ the \textbf{internal morphism functor}, and if it is representable by an $S$-scheme, we denote the representable scheme by $X{}^{\wedge}_S Y$ and call it the \textbf{internal morphism space}. When $Y=\Spec(B)$ and $S=\Spec(A)$ we sometimes write $X_{B/A}\coloneqq X{}^{\wedge}_{\Spec(A)}\Spec(B)$.
\end{definition}

\begin{remark}\label{rem: int hom weil restrict}
    Notice, that the internal morphism functor is equivalent to a Weil restriction of the fibered product, i.e.$$
        \underline{X{}^{\wedge}_{S}Y} \cong \Res_{Y/S}(X\times_SY).$$
\end{remark}

\begin{lemma}\label{lem: internal hom repr}
    If $Y$ is finite and {\'e}tale, and $X$ is $S$-affine then $\underline{X{}^{\wedge}_{S}Y}$ is representable by an affine scheme of finite type.
\end{lemma}
\begin{proof}
    By Remark \ref{rem: int hom weil restrict}, $
        \underline{X{}^{\wedge}_{S}Y} \cong \Res_{Y/S}(X\times_SY)$ and by \cite[Proposition 4.4]{scheiderer2006real}, as $Y$ is finite locally free over $S$ (see \cite[Definition 29.48.1]{SP}) and $X\times_s Y$ is affine, $\Res_{Y/S}(X\times_SY)$ is representable by an affine scheme. Moreover, by \cite[Lemma 3.5(2)]{AA22} it is of finite type.
\end{proof}

\begin{lemma}\label{lem: internal hom base change}
    The internal morphism space is preserved under base change. Namely, let $S'$ and $Y$ be $S$-schemes. Denote $X'\coloneqq X\times_S S'$ and $Y'\coloneqq Y\times_S S'$.  If $\underline{X{}^{\wedge}_S Y}$ is representable by an $S$-scheme $X{}^{\wedge}_S Y$, then  $\underline{X'{}^{\wedge}_{S'} Y'}:\mathbf{Sch}_{S'}\to\mathbf{Set}$ is representable by $(X{}^{\wedge}_S Y)\times_{S} S'$ as an $S'$-scheme.
\end{lemma}
\begin{proof}
    Let $T$ be an $S'$-scheme. By definition,  an $S'$-morphism $x:T\to(X{}^{\wedge}_S Y)\times_{S} S'$, is uniquely defined by its composition with the projection on $p_{(X{}^{\wedge}_S Y)}:(X{}^{\wedge}_S Y)\times_{S} S'\to X{}^{\wedge}_S Y$. Notice, that by representability, $p_{(X{}^{\wedge}_S Y)}\circ x$ induce a unique $S$-morphism $x':Y\times_S T \to X$. Now, we can base change from $S$ to $S'$ and get an $S'$-morphism $x'': (Y\times_S T)\times_S S' \to X\times S'$. Notice, that as
    $$ (Y\times_S T)\times_S S'\cong  (Y\times_S S')\times_{S'} T=Y'\times_{S'} T$$
    we get an $S'$-morphism $ Y'\times_{S'} T \to X'$, hence an element in $\underline{X'{}^{\wedge}_{S'} Y'}(T)$.

    On the other hand, let $y\in\underline{X'{}^{\wedge}_{S'} Y'}(T)$, i.e. an $S'$-morphism $y:Y'\times_{S'}T\to X'$. Notice, that
    \[
        Y'\times_{S'} T=(Y\times_S S')\times_{S'} T\cong Y\times_S T,
    \]
    hence we got an $S$-morphism $Y\times_{S}T\to X\times_S S'$, and composing with the projection $X\times_S S'\to X$, and using representability, we get an $S$-morphism $y':T \to X{}^{\wedge}_S Y$. Now, $y'$ together with the $S'$-structure morphism of $T$, induce an morphism $T\to (X{}^{\wedge}_S Y)\times_{S} S'$.

    We have defined a bijection between $\underline{X'{}^{\wedge}_{S'} Y'}(T)$ and $\underline{(X{}^{\wedge}_S Y)\times_{S} S'}(T)$ which is functorial in $T$, hence an equivalence of categories.
\end{proof}

\begin{lemma}\label{lem: etale extent restrict}
    Let $\Gamma$ be a finite group, $F$ a field, $E$ a finite {\'e}tale $F$-algebra endowed with a $\Gamma$-action  over $F$, and $Y$ an affine $F$-scheme. Then the functors $ \underline{(Y_{E/F})^\Gamma}$, $\underline{Y_{E^\Gamma/F}}$ are representable by the same affine scheme.
\end{lemma}
\begin{proof}
    First, notice that by Lemma \ref{lem: internal hom repr}, $\underline{Y_{E^\Gamma/F}}$ and $\underline{Y_{E/F}}$ are representable by affine schemes. By Lemma \ref{lem: fixed points scheme}, $\underline{(Y_{E/F})^\Gamma}$ is also representable by an affine scheme. Therefore, we are left to show that  $ \underline{(Y_{E/F})^\Gamma}$ and $\underline{Y_{E^\Gamma/F}}$ are equivalent.

    Let $T$ be an $F$-scheme. Notice, that an element in $ \underline{(Y_{E/F})^\Gamma}(T)$ is an $F$-morphism $x:T\to Y_{E/F}$, fixed by composition with the $\Gamma$-action on $Y_{E/F}$, where $Y_{E/F}$ is the affine scheme representing $\underline{Y_{E/F}}$. By representability, $x$ is the same as a morphism $x':T\times_{\Spec(F)}\Spec(E)\to Y$, fixed by pre-composition with the $\Gamma$-action induced on $T\times_{\Spec(F)}\Spec(E)$. Therefore, $x'$ can be considered as a map $T\times_{\Spec(F)}\Spec(E)/\Gamma\to Y$, as can be checked by passing to the functor of points. By Lemma \ref{lem: qoutient of etale}, $\Spec(E)/\Gamma\cong\Spec(E^\Gamma)$. In conclusion, we get a morphism $x'':T\times_{\Spec(F)} \Spec(E^\Gamma)\to X$, which corresponds to a unique element in $\underline{Y_{E^\Gamma/F}}(T)$.

    This construction is reversible, hence forms a bijection between $ \underline{(Y_{E/F})^\Gamma}(T)$ and $\underline{Y_{E^\Gamma/F}}(T)$ which is functorial in $T$. Therefore, we have an equivalence of functors between  $ \underline{(Y_{E/F})^\Gamma}$ and $\underline{Y_{E^\Gamma/F}}$.
\end{proof}

\subsection{Forms of reductive groups}

Throughout this subsection, we fix a field $F$ and a connected reductive $F$-group $\mathbf{G}$. We review the Galois cohomology approach to classification of $F$-forms of $\mathbf{G}$ (see \cite[Section \S 1]{Con14} for a full exposition). Additionally, we prove a few lemmas that follow from this approach, formulated in a way that is convenient for our purposes.

\begin{definition}
    Let $\mathbf{H}$ be another reductive $F$-group. We say that $\mathbf{H}$ is an $\mathbf{F}$\textbf{-form} of $\mathbf{G}$ if for some finite Galois extension $E/F$, we get an isomorphism $\mathbf{H}_E\cong\mathbf{G}_E$.
\end{definition}

Before proceeding with our review of forms, we define the rank of a reductive group scheme, which is an invariant of forms. This definition is used in the main result (Theorem \ref{thm: versal family reductive}) and in several other constructions.

\begin{definition}
    The \textbf{rank of a root datum} $(X,\Phi,X^\vee,\Phi^\vee)$ (see e.g. \cite[Definition 1.3.3]{Con14}) is the rank of $X$ as a free $\bZ$-module. The \textbf{absolute root datum} of $\mathbf{G}$ is the root datum associated to $\mathbf{G}_{\overline{F}}$ (see e.g.  \cite[Theorem 1.3.15]{Con14}), where $\overline{F}$ is the algebraic closer of $F$.  The \textbf{rank of a reductive group} $\mathbf{G}$ is the rank of its absolute root datum.
\end{definition}

\begin{remark}
    The absolute root datum of $\mathbf{G}_E$ is constant for any finite Galois extension $E/F$. In particular, all of $F$-forms of $\mathbf{G}$ has the same absolute root data, hence the same rank.
\end{remark}

We continue with our discussion of classification of $F$-forms by Galois cohomology, via the definition of  1-cocycles.

\begin{definition}
    Let $\Gamma$ be a group acting on $\mathbf{G}$ over $F$. Define a $\Gamma$-action on $\Aut_F(\mathbf{G})$ by $\prescript{\gamma}{}{\phi}=\gamma\circ\phi\circ \gamma^{-1}$, for any $\phi\in\Aut_F(\mathbf{G})$ and $\gamma\in\Gamma$. A \textbf{1-cocycle} is a map $c:\Gamma\to\Aut_F(\mathbf{G})$ satisfying \[
        c(\gamma_1\gamma_2)=c(\gamma_1)\cdot \prescript{\gamma_1}{}{c(\gamma_2)}
    \]
    for any $\gamma_1,\gamma_2\in\Gamma$.
\end{definition}

We denote the set of all 1-cocycle by $Z^1(\Gamma,\Aut_F(\mathbf{G}))$, which can be consider as a functor from $\mathbf{Sch}_F$ to  $\mathbf{Set}$. We later show, that under certain assumptions, this functor is representable (see Lemma \ref{lem: 1-cocycles scheme}). 1-cocycles naturally appears in the following case.

\begin{definition}\label{def: induced cocycle}
    Let $\Gamma$ be a group, $E$ an $F$-algebra endowed with a $\Gamma$-action on it, and $\mathbf{G}'$ another reductive $F$-group. Notice, that there is a natural action of $\Gamma$ on $\mathbf{G}_E$ and $\mathbf{G}'_E$. For any isomorphism $a:\mathbf{G}_E\xrightarrow{\sim}\mathbf{G}'_E$, define a 1-cocycle $c_a\in Z^1(\Gamma,\Aut_F(\mathbf{G}_E))$ by
    \[
        c_a(\gamma)=a^{-1}\circ\gamma\circ a\circ \gamma^{-1},
    \] for any $\gamma\in\Gamma$.  We call $c_a$ the \textbf{induced 1-cocycle} from $a$.
\end{definition}

Notice, that if $a$ above was $\Gamma$-equivariant then, $c_a(\gamma)=id_{\mathbf{G}_E}$. Therefore, we informally say that $c_a$ measures the obstruction of $a$ from being $\Gamma$-equivariant. An important case of Definition \ref{def: induced cocycle}, is when $\mathbf{G}'$ is an $F$-form of $\mathbf{G}$, $E/F$ is a Galois extension such that $\mathbf{G}_E\cong \mathbf{G}'_E$ and  $\Gamma=\Gal(E/F)$. Therefore, we found a map between $F$-forms and 1-cocycles. We now show how to use this map to get a full classification of $F$-forms of $\mathbf{G}$.

\begin{definition}
    Let $\Gamma$ be a group acting on $\mathbf{G}$ over $F$. Define an equivalence relation on $Z^1(\Gamma,\Aut_F(\mathbf{G}))$ as follows: two 1-cocycles  $c_1,c_2\in Z^1(\Gamma,\Aut_F(\mathbf{G}))$ are equivalent if there is some $\phi\in\Aut_F(\mathbf{G})$ such that $c_1(\gamma)=\phi^{-1}\circ c_2(\gamma)\circ \prescript{\gamma}{}{\phi}$  for all $\gamma\in \Gamma$.

    We say that such $c_1,c_2$ are \textbf{cohomologous} or that they satisfy the \textbf{1-coboundary} condition. The set of all equivalence classes is called the \textbf{1-cohomology classes}  and denoted by $H^1(\Gamma,\Aut_F(\mathbf{G}))$.
\end{definition}

As mentioned above, these definition are used in the cohomological approach to classifying $F$-forms of $\mathbf{G}$. This approach is formulated in the following lemma, which is a rephrasing of \cite[Lemma 7.1.1]{Con14}, stated for finite Galois extension instead of the absolute Galois extension. The proof in \cite{Con14} prove this formulation as well.

\begin{lemma}[cf. {\cite[Lemma 7.1.1]{Con14}}]\label{lem: cohomo cocycles iff iso forms}
    Let $\mathbf{G}'$ and $\mathbf{G}''$ be reductive $F$-groups, and $E/F$ be a finite Galois extension such that: \[
        \mathbf{G}_E\cong\mathbf{G}'_E\text{ and } \mathbf{G}_E\cong\mathbf{G}''_E
    \]
    The induced 1-cocycles, $c',c''\in Z^1(\Gal(E/F),\Aut_F(\mathbf{G}_E))$, are cohomologous if and only if $\mathbf{G}'\cong\mathbf{G}''$ over $F$.

    In other words, $H^1(\Gal(E/F),\Aut_F(\mathbf{G}_E))$ classify all $F$-forms of $\mathbf{G}$ that becomes isomorphic to $\mathbf{G}_E$ after base change to $E$.
\end{lemma}

Finally, we show how to generalize the classification presented in Lemma \ref{lem: cohomo cocycles iff iso forms} from Galois extension, to any {\'e}tale $F$-algebra $E$, endowed with a group action of some finite group $\Gamma$, such that $E^\Gamma=F$. To do so, we first show how to get a group action from an $1$-cocycle.

\begin{definition}\label{def: cocycle to action}
    Let $\Gamma$ be a group acting on $\mathbf{G}$ over $F$. For any $c\in Z^1(\Gamma,\Aut_F(\mathbf{G}))$ we define a $\Gamma$-action $*_c:\Gamma\to\Aut_F(\mathbf{G})$ as $*_c(\gamma)=c(\gamma)\circ\gamma$ for any $\gamma \in \Gamma$.
\end{definition}

Notice, that any two 1-cocycles $c_1,c_2\in Z^1(\Gamma,\Aut_F(\mathbf{G}))$ are cohomologous if and only if there is some element $\phi\in\Aut_F(\mathbf{G})$, which is $\Gamma$-equivariant with respect to the $\Gamma$-action $*_{c_1}$ on the source and $*_{c_2}$ on the target.

\begin{lemma}\label{lem: etale extension reduction}
    Let $\Gamma$ be a finite group, $E$ a finite {\'e}tale $F$-algebra endowed with a $\Gamma$-action on $E$, such that $E^\Gamma=F$ and $c\in Z^1(\Gamma, \Aut_F(\mathbf{G}_{E/F}))$ a 1-cocycle. Then the functor $\underline{\mathbf{G}_{E/F}^{*_c(\Gamma)}}$ (see Definition \ref{def: fixed points/quotient functor/space}) is representable by an $F$-form of $\mathbf{G}$, denoted by $\mathbf{G}_{E/F}^{*_c(\Gamma)}$. In other words, for some finite Galois extension  $E'/F$,
    \begin{equation}\label{eq: etale extension reduction}
        \mathbf{G}_{E'}\cong (\mathbf{G}_{E/F}^{*_c(\Gamma)})_{E'}.
    \end{equation}

    Moreover, if $E/F$ is a finite Galois extension, then we may take $E'=E$ and choose the isomorphism in Equation \eqref{eq: etale extension reduction} to induce $c$.
\end{lemma}
\begin{proof}
    The main claim we need to prove is the isomorphism in Equation \ref{eq: etale extension reduction}. This is because representability follows from Lemma \ref{lem: fixed points scheme}, and the Galois extension case will be immediate from our proof of the general case.

    By the characterization of finite {\'e}tale $F$-algebras $E\cong \prod_{i=1}^n E_i$, where each $E_i$ is a separable extension of $F$. Set $E'$ to be the minimal Galois extension of $F$ such that $E_i$ embeds in $ E'$ for all $i\leq n$. Now, let $m\coloneqq\sum_{i=1}^n [E_i/F]$ and notice that:
    \begin{align*}
        \Spec(E)\times_{\Spec(F)}\Spec(E')
         & \cong\Spec(E\otimes_{F}E')                   \\
         & \cong\Spec((\prod_{i=1}^n E_i)\otimes_{F}E') \\
         & \cong\Spec(\prod_{i=1}^n (E_i\otimes_F E'))  \\
         & \cong\Spec(\prod_{i=1}^m (E'))               \\
         & \cong\bigsqcup_{i=1}^m\Spec(E').
    \end{align*}
    Moreover, $\Spec(E)\times_{\Spec(F)}\Spec(E')$ inherits a $\Gamma$-action over $E'$, which must act by permutations on the copies of $\Spec(E')$. In addition, notice that $(E\otimes_{F}E')^\Gamma=E'$, hence $\Gamma$ acts transitively of the set on copies of $\Spec(E')$.

    Now, by Lemma \ref{lem: internal hom base change} we get
    $$(\mathbf{G}_{E/F})_{E'}\cong (\mathbf{G}_{E'})_{E\otimes E'/E'}\cong (\mathbf{G}_{E'})_{\prod_{i=1}^m (E')/E'}\cong \prod_{i=1}^m\mathbf{G}_{E'},$$
    and the $\Gamma$-action transitively permutes the copies of $\mathbf{G}_{E'}$, as can be directly checked on the level of the associated functor of points.

    Now, let $c_{E'}\in Z^1(\Gamma,\Aut_{E'}((\mathbf{G}_{E/F})_{E'}))$ be the induced 1-cocycle from $c$ via base change by $\Spec(E')$. We get that the $*_{c_{E'}}$-action is also transitive on the $m$-copies, by considering $c_{E'}^{-1}$. Therefore, by Lemma \ref{lem: limit base change} $$(\mathbf{G}_{E/F}^{*_c(\Gamma)})_{E'}\cong((\mathbf{G}_{E/F})_{E'})^{*_{c_{E'}}(\Gamma)}\cong\mathbf{G}_{E'}.$$
\end{proof}
We may add a few assumptions to Lemma \ref{lem: etale extension reduction}, and verify that the $F$-form obtained from it is quasi-split (see the definition in \cite[15.4]{Milne_reductive_groups}).

\begin{lemma}\label{lem: etale extension reduction - quasi-split}
    In the settings of Lemma \ref{lem: etale extension reduction}, assume that $\mathbf{G}$ is split an choose a pinning $(\mathbf{G}, \mathbf{T}, \mathbf{B}, \{X_a\}_{a\in\Delta})$ for it (see \cite[Definition 1.5.4]{Con14}). Moreover, assume the image of the 1-cocycle $c$ lies in $\Aut(\mathbf{G}, \mathbf{T}, \mathbf{B}, \{X_a\}_{a\in\Delta})$, i.e., preserve the pinning. Then the reductive group $\mathbf{G}_{E/F}^{*_c(\Gamma)}$ in Lemma \ref{lem: etale extension reduction} is quasi-split. \end{lemma}
The conventional proof in the case where $E/F$ is a finite Galois extension works in this settings as well. Let us be more precise.
\begin{proof}
    First, notice that $\mathbf{B}_{E/F}\subseteq\mathbf{G}_{E/F}$ is a Borel subgroup.  Recall that $\mathbf{G}_{E/F}^{*_c(\Gamma)}$ was obtained by taking the fixed points of $\mathbf{G}_{E/F}$ under the action $*_c$. By our assumption, $*_c$ preserve the Borel group $\mathbf{B}_{E/F}$. By \cite[15.4]{Milne_reductive_groups}, in order to show that $\mathbf{B}_{E/F}^{*_c(\Gamma)}\subseteq \mathbf{G}_{E/F}^{*_c(\Gamma)}$ is Borel, it is enough to show it after base change by some field extension. The field extension $E'$ as in the proof of Lemma \ref{lem: etale extension reduction} shows that as
    $$(\mathbf{B}_{E/F}^{*_c(\Gamma)})_{E'}\cong \mathbf{B}_{E'} \text{ and } (\mathbf{G}_{E/F}^{*_c(\Gamma)})_{E'}\cong \mathbf{G}_{E'},$$
    by the same isomorphism.
\end{proof}

\subsection{Reductive group schemes}
In this subsection, we review some of the definitions used for the classification of reductive group schemes, which are essential for our proof. These definitions are taken almost verbatim from \cite{Con14}, which give a full description of the classification method. Our primary goal in this subsection, is to clarify how the terminology used in the classification of reductive groups over fields, which we refer to as the classical theory,  generalizes to the relative setting, over a general non-empty scheme $S$.

Throughout this subsection, fix a non-empty scheme $S$. The first and most natural place to start, is to generalize the notion of reductive groups over a field to reductive group schemes, over any non-empty scheme.

\begin{definition}[see {\cite[Definition 3.1.1]{Con14}}]\label{def: reductive over S}
    A \textbf{reductive} $\mathbf{S}$\textbf{-group} is a smooth $S$-affine group scheme $G\to S$ such that the geometric fibers $G_{\overline{s}}$ are \textit{connected} reductive groups.
\end{definition}

In the classical theory, reductive groups are often permitted to be disconnected. In the beginning of \cite[section $\S$3]{Con14}, there is a full discussion of why it is convenient to assume connectedness in the relative case. We move on to generalize the notion of a torus.

\begin{definition}[see {\cite[Appendix B]{Con14}}]
    For a scheme $S$ and finitely generated $\bZ$-module $M$, we define $D_S(M)$ to be the $S$-group $\Spec(\cO_S[M])$, representing the functor $\Hom_{S-gp}(M_S, G_m)$ of characters of the constant $S$-group $M_S$.
\end{definition}

\begin{definition}[see {\cite[B.1.1]{Con14}}]
    A group scheme $G\to S$ is of \textbf{multiplicative type} if there is an fppf covering $\{S_i\}$ of $S$ such that $G_{S_i}\cong D_{S_i}(M_i)$ for a finitely generated abelian group $M_i$ for each $i$. Such a $G$ is \textbf{split} (or \textbf{diagonalizable}) over $S$ if $G \cong D_S(M)$ for some $M$.
\end{definition}

\begin{definition}[see {\cite[Definition 3.1.1]{Con14}} and {\cite[Definition 3.2.1]{Con14}}]
    An $\mathbf{S}$\textbf{-torus} is an $S$-group $T\to S$ of multiplicative type with smooth connected fibers.  A \textbf{maximal torus} in a smooth $S$-affine group scheme $G\to S$ is a torus $T\subseteq G$ such that for each geometric point $\overline{s}$ of $S$ the fiber $T_{\overline{s}}$ is not contained in any strictly larger torus in $G_{\overline{s}}$.
\end{definition}

Now, we generalize the notion of roots space.

\begin{definition}[see {\cite[Definition 4.1.1]{Con14}}]\label{def: root space}
    Assume that $G$ admits a split maximal torus $T$ over $S$ and fix an isomorphism $T\cong D_S(M)$ for a finite free $\bZ$-module $M$. The T-action on $\fg = \Lie(G)$ then corresponds to an $\cO_S$-linear $M$-grading $\bigoplus_{m\in M} \fg_m$ of the vector bundle $\fg$, where $t \in T$ acts on the subbundle $\g_m$ via multiplication by the unit $m(t)$.

    A \textbf{root} for $(G, T)$ is a nonzero element $a \in M$ such that $\fg_a$ is a line bundle. We call such $\fg_a$ a \textbf{root space} for $(G, T, M)$.
\end{definition}

\begin{definition}[see {\cite[Definition 5.1.1]{Con14}}]
    A reductive $S$-group $G$ is \textbf{split}, if there exists a split maximal torus $T$ equipped with an isomorphism $T \cong D_S(M)$ for a finite free $\bZ$-module $M$, such that:
    \begin{enumerate}
        \item the nontrivial weights, $a: T \to \mathbf{G}_m$ that occur on $\fg = \Lie(G)$ arise from elements of $M$ (so in particular, such a is a root for $(G, T)$ and is “constant sections” of $M_S$),
        \item each root space $\fg_a$ is free of rank 1 over $\cO_S$,
        \item   each coroot $a^{\vee}:\mathbf{G}_m \to T$ arises from an element of the dual lattice $M^{\vee}$ (i.e., $a^{\vee}$ as a global section of $M_S^{\vee}$ over $S$ is a constant section).
    \end{enumerate}
\end{definition}

Finally, we generalize the notions of parabolic subgroups, Borel subgroups and pinning.

\begin{definition}[see {\cite[Definition 5.2.1]{Con14}}]\label{def: parabolic}
    A \textbf{parabolic subgroup} of a reductive group scheme $G\to S$ is a smooth $S$-affine group scheme $P$, equipped with a monic homomorphism $P \to G$, $P_{\overline{s}}$ is parabolic in $G_{\overline{s}}$ (i.e., $G_{\overline{s}}/P_{\overline{s}}$ is proper) for all geometric points $\overline{s}$ of $S$.  By \cite[Proposition 5.2.3]{Con14}  $P \to G$ is a closed immersion.
\end{definition}

\begin{definition}[see {\cite[Definition 5.2.10]{Con14}}]\label{def: Borel, quasi-split}
    A \textbf{Borel subgroup} of a reductive group scheme $G\to S$ is a parabolic subgroup $B\subseteq G$ such that $B_{\overline{s}}$ is a Borel subgroup of $G_{\overline{s}}$ for all geometric points $\overline{s}$ of $S$. A reductive group $G\to S$ is \textbf{quasi-split} over $S$ if it admits a Borel subgroup scheme over $S$.
\end{definition}

\begin{definition}[see {\cite[Definition 6.1.1]{Con14}}] \label{def: pinning}
    Let $(G, T, M)$ be a split reductive group over a nonempty scheme $S$, and let $R(G, T, M) = (M, \Phi, M^\vee, \Phi^\vee)$ be its associated root datum. A \textbf{pinning} on $(G, T, M)$ is a pair $(\Phi^+, \{X_a\}_{a\in\Delta})$ consisting of a positive system of roots $\Phi^+\subseteq \Phi$ (or equivalently, a base $\Delta$ of $\Phi$) and trivializing sections $X_a \in \fg_a(S)$ for each simple positive root $a \in \Delta$.
    The 5-tuple $(G,T,M,\Phi^+, \{X_a\}_{a\in\Delta})$ is a \textbf{pinned split reductive} $\mathbf{S}$\textbf{-group}.
\end{definition}

\begin{remark}\label{rem: pinning grp scheme / grp}
    Notice, that Definition \ref{def: pinning} is a generalization of the definition of a pinning on a split reductive group over a field, as in \cite[Definition 1.5.4]{Con14}.
\end{remark}

\section{Proof of main results}\label{sec: proof of main results}

\subsection{Family of finite {\'e}tale algebras}

The first step in our proof is to construct a finite {\'e}tale morphism $\Psi_{\Gamma} : \cE_{\Gamma} \rightarrow  \cF_{\Gamma}$ for any finite group $\Gamma$. The main feature of this morphism is that the collection of its fibers contains any Galois extensions $\Spec(E)\to\Spec(F)$, with Galois group $\Gal(E/F)=\Gamma$.

\begin{lemma} \label{lem: fam of finite etale algebras}
    Let $\Gamma$ be a finite group. There exists a finite {\'e}tale morphism $\Psi_{\Gamma} : \cE_{\Gamma} \rightarrow  \cF_{\Gamma}$ of  finitely presented affine $\bZ$-schemes, with a $\Gamma$-action on $\cE_{\Gamma}$ over $\cF_{\Gamma}$, such that for any field $F$:
    \begin{enumerate}
        \item \label{lem: fam of finite etale algebras: 1}For any point $x\in \cF_\Gamma(F)$, the functor $\underline{(\cE_\Gamma)_x/\Gamma}$  (see definition \ref{def: fixed points/quotient functor/space}) is representable by $\Spec(F)$.
        \item \label{lem: fam of finite etale algebras: 2}For any finite Galois extension $E/F$ with Galois group $\Gamma = \Gal(E/F)$, there exists a point $\nu \in \cF_{\Gamma}(F)$ such that $(\cE_\Gamma)_\nu\cong\Spec(E)$ and the action of  $\Gamma$ on $\Spec(E)$ is the Galois action.
    \end{enumerate}
\end{lemma}

In order to prove Lemma \ref{lem: fam of finite etale algebras}, we first need to prove a few lemmas.

\begin{lemma}\label{lem: locally constant fibers}
    Let $\phi:X\to S$ be a finite {\'e}tale morphism, and $k$ an algebraically closed field. Define $n_k:S(k)\to \bN$ as $n(\overline{s})=|X_{\overline{s}}|$. Then $n_k$ factors through a function $n:S\to\bN$ which is locally constant.
\end{lemma}
\begin{proof}
    By \cite[Exercise 25.2.C]{Vak17} $\phi$ is finite, flat, and locally of finite presentation, and by \cite[Lemma 29.48.2]{SP}, it is finite and locally free. Therefore, for any point $s\in S$ there is an open neighborhood $U$ such that $(\phi_*\cO_X)|_U$ is free over $\cO_S|_U$. Then the stalk of $\phi_*\cO_X$ at every point $s\in U$ is free with the same rank over the stalk of $\cO_S$ at $s$. We define $n(s)$ to be this rank.

    Now, we define a map $S(k)\to S$ by mapping a morphism $\overline{s}\in S(k)$ to its image, which we denote by $s\in S$. This gives us the required factorization as $$X_{\overline{s}}\cong \Spec ((\phi_*\cO_X)_s\otimes_{(\cO_S)_s} k)\cong \Spec(k^{n(s)})\cong \bigsqcup_{i=1}^{n(s)}\Spec(k)$$
\end{proof}
\begin{lemma}\label{lem: connect comp reduction}
    Let $\Gamma$ be a finite group, $X$, $Y$ be affine finitely presented $\bZ$-schemes, and $\psi:X\to Y$ a finite {\'e}tale morphism. Endow $X$ with a $\Gamma$-action over $Y$. Let $Y'\subseteq Y$ be a union of connected components of $Y$, $X'\coloneqq \psi^{-1}(Y')$ and $\psi'=\psi|_{X'}$. Then $X'$ is also a union of connected components of $X$, and $X'$,  $Y'$ and $\psi'$ preserve all the properties we described for $X$, $Y$ and $\psi$.
\end{lemma}
\begin{proof}
    First, the facts that $X'$ is also a union of connected components of $X$ and that for any $\gamma\in \Gamma$ $\gamma(X')= X'$ over $Y'$ are topological in nature and can be easily checked.

    Now, as $X$ is finitely presented $\bZ$-scheme it is Noetherian. Therefore, we can write $X=\bigsqcup_{i=1}^n C_i$, where $C_i$ are connected components. Without loss, we may assume $X'=\bigsqcup_{i=1}^m C_i$ for some $m\leq n$.

    Any connected component is a closed subscheme of an affine scheme, hence affine. Therefore, we can write $C_i=\Spec(A_i)$, $X'=\Spec(\prod_{i=1}^m A_i)$ and $X=\Spec(\prod_{i=1}^n A_i)$. Accordingly, $X'$ is finitely presented over $X$, hence over $\bZ$. The same can be shown for $Y'$.

    For the remainder of the properties, notice that: \[
        \begin{tikzcd}
            X' \arrow[r] \arrow[d] & X \arrow[d] \\
            Y' \arrow[r] & Y
        \end{tikzcd}
    \]
    is Cartesian, and finiteness and {\'e}taleness are preserved under base change.
\end{proof}
We return to the proof of Lemma \ref{lem: fam of finite etale algebras}. The main idea of our proof is to emulate the properties of Galois extensions in scheme theoretic tools.

\begin{proof}[Proof of Lemma \ref{lem: fam of finite etale algebras}]
    Let $m = |\Gamma|$, and enumerate $\Gamma$ as $\{g_1, \ldots, g_{m}\}$. For all $i, j \leq m$, there exists $k \leq m$ such that $g_k = g_i g_j$. We denote this by $\mu_\Gamma(i, j) \coloneqq k$.

    Next, define a functor $\underline{\cF_\Gamma'} : \mathbf{Ring} \to \mathbf{Set}$ as follows: for a ring $A$, $\underline{\cF_\Gamma'}$ consists of tuples $(f, (h_i)_{i\leq m}, (d_i)_{i\leq m}, (e_{i,j})_{i,j\leq m}, z)$, where:
    \begin{enumerate}
        \item $f$ is a monic polynomial of degree $m$ with coefficients in $A$.
        \item $(h_i)_{i\leq m}$ is a sequence of  polynomials of degree at most $m-1$, with coefficients in $A$ (corresponding to condition \ref{proof: fam of finite etale algebras: b} below).
        \item $(d_i)_{i\leq m}$ is a sequence of polynomials of degree at most $m(m-1) - m$, with coefficients in $A$ (corresponding to condition \ref{proof: fam of finite etale algebras: b} below).
        \item \label{proof: fam of finite etale algebras: 4} $(e_{i,j})_{i,j\leq m}$ is a sequence of polynomials of degree at most $(m-1)^2 - m$, with coefficients in $A$ (corresponding to condition \ref{proof: fam of finite etale algebras: c} below).
    \end{enumerate}
    In cases where the bound on the degree becomes $-1$, the corresponding polynomial must be the zero polynomial.

    These tuples must satisfy the following conditions:
    \begin{enumerate}[label=\alph*.]
        \item \label{proof: fam of finite etale algebras: a} $\text{res}(f, f') \in A^\times$, meaning $f$ is separable;
        \item \label{proof: fam of finite etale algebras: b} For all $i \leq m$, $f \cdot d_i = f \circ h_i$, meaning that $h_i$ preserves the roots of $f$;
        \item \label{proof: fam of finite etale algebras: c} For all $i, j \leq m$, $f \cdot e_{i,j} = h_i \circ h_j - h_{\mu_\Gamma'(i,j)}$, meaning that the polynomials $(h_i)_{i\leq m}$ act as $\Gamma$ on the roots of $f$.
    \end{enumerate}

    Next, define a functor $\underline{\cE_\Gamma'} : \bf{Ring} \to \bf{Set}$ by augmenting $\underline{\cF_\Gamma'}(A)$ with an additional element $z \in A$ such that $f(z) = 0$. Define a $\Gamma$-action on $\underline{\cE_\Gamma'}(A)$ by:
    \[
        g_i \cdot (f, (h_i)_{i\leq m}, (d_i)_{i\leq m}, (e_{i,j})_{i,j\leq m}, z) \mapsto (f, (h_i)_{i\leq m}, (d_i)_{i\leq m}, (e_{i,j})_{i,j\leq m}, h_i(z)).
    \]

    We show that $\underline{\cE_\Gamma'}$ and $\underline{\cF_\Gamma'}$ are representable by finitely presented affine $\bZ$-schemes, denoted by the $\cF_\Gamma'$ and $\cE_\Gamma'$, respectively. Consider the polynomial ring over $\bZ$ with variables corresponding to the coefficients of the polynomials described in the definition of $\underline{\cF'_\Gamma}$. This ring of integer multi-variables polynomials, is then quotiented by the ideal generated by the polynomials describing conditions \ref{proof: fam of finite etale algebras: a}, \ref{proof: fam of finite etale algebras: b} and \ref{proof: fam of finite etale algebras: c}.

    The construction in the last paragraph, allows us to write $\cF_\Gamma' \coloneqq \Spec\,R$, and $\cE_\Gamma' \coloneqq \Spec\,R[z]/(f(z))$. Therefore, the natural projection $\Psi'_{\Gamma} : \cE'_{\Gamma} \to \cF'_{\Gamma}$ is {\'e}tale by definition (see e.g. \cite[Definition 12.6.2]{Vak17}) and by condition \ref{proof: fam of finite etale algebras: a}. Moreover, it is finite as $R[z]/(f(z))$ is a finite $R$-module. In addition   the action of $\Gamma$ on $\cE'_\Gamma(A)$ is functorial in $A$ over $\cF'_\Gamma$ and the schemes are affine, so we have an action of $\Gamma$ on $\cE'_\Gamma$ over $\cF'_\Gamma$ (see Definition \ref{def: group action on  a scheme} and the remark following it).

    Now, by Lemma \ref{lem: qoutient of etale} $\underline{\cE'_\Gamma/\Gamma}$, is representable by $\Spec ((R[z]/(f(z)))^\Gamma)$. Let $n$ be as in Lemma \ref{lem: locally constant fibers}, and define $\cF_\Gamma\subseteq \cF_\Gamma'$ to be $n^{-1}(1)$, $\cE_\Gamma\coloneqq\Psi_\Gamma'^{-1}(\cF_\Gamma)\subseteq \cE_\Gamma'$, and $\Psi_\Gamma\coloneqq \Psi_\Gamma'|_{\cE_\Gamma}$. By Lemma \ref{lem: connect comp reduction}, $\Psi_{\Gamma} : \cE_{\Gamma} \rightarrow  \cF_{\Gamma}$ keeps all the required properties. Moreover, all of our construction are preserved under base change, hence clause \eqref{lem: fam of finite etale algebras: 1} follows.

    Now, for clause \eqref{lem: fam of finite etale algebras: 2}, we construct the point $\nu \in \cF_\Gamma(F)$ for any Galois extension $E/F$, such that $\Gamma=\Gal(E/F)$. We first construct $\nu$ as an element of $\cF_\Gamma'(F)$ and then show its image lies in $\cF_\Gamma$, hence it is also an element of $\cF_\Gamma(F)$.

    By the primitive element theorem, there is an $\alpha \in E$ such that $E = F(\alpha)$. We construct a tuple $(f^0, (h^0_i)_{i\leq m}, (d^0_i)_{i\leq m}, (e^0_{i,j})_{i,j\leq m})$ of polynomials in $F[x]$.
    \begin{enumerate}
        \item $f^0$ is the minimal (monic) polynomial of $\alpha$ over $F$.
        \item Let $\{\beta_i\}_{i\leq m}$ be the roots of $f^0$ in $E$, with $\beta_0 = \alpha$. Because each $\beta_i \in E = F(\alpha)$, and $E/F$ is a degree $m$ extension, then $\beta_i=h_i^0(\alpha)$ for polynomials $h_i^0(x)\in F[x]$ of degree $< m$.
        \item Since $f^0 \circ h^0_i(\alpha) = 0$ for all $i\leq m$, we have $f^0 \mid f^0 \circ h^0_i$. Therefore, there exists polynomials $d_i^0(x)\in F[x]$, such that $f^0\cdot d_i^0 = f^0\circ h_i^0$. This construction of $d_i^0$ also ensures that $\deg(d_i^0) < m(m-1) - m$.
        \item Similarly, construct $(e^0_{i,j})_{i,j \leq m}$ to satisfy condition \ref{proof: fam of finite etale algebras: c}.
    \end{enumerate}
    By representability this tuple corresponds to a point $\nu\in\cF'_{\Gamma}(F)$.

    We verify that
    \[
        (\cE'_\Gamma)_\nu\cong\Spec(E),
    \]
    as required. We are working with affine schemes, therefore it is enough to show that\[
        F\otimes_R R[z]/(f(z))\cong E,
    \]
    which is clear as
    \[
        F\otimes_{R} R[z]/(f(z))\cong F[z]/(f(z))\cong F(\alpha)\cong E.
    \]
    Moreover, our construction ensures that $\{h^0_i\}_{i<n}$ acts as the Galois action. Finally, notice that $\nu\in\cF_\Gamma(F)$ as $n(Im(\nu))=\dim_F(E^\Gamma)=1$.
\end{proof}

\subsection{Versal family of quasi-split reductive groups}
The next part of our proof is to construct a versal family of quasi-split reductive groups. First, we construct such a versal family when the root data of the quasi-split reductive groups are fixed.

\begin{lemma} \label{lem: scheme quasi-split fixed r.d.}
    Let $\fX$ be a root datum. There is a scheme $\cF_\fX$  of finite type and a smooth $\cF_\fX$-affine group scheme of finite type $\cH_\fX \to\cF_\fX$ such that:
    \begin{enumerate}
        \item \label{lem: scheme quasi-split fixed r.d.: 1}For every field $F$ and every $x\in\cF_\fX(F)$, the group $(\cH_\fX)_x$ is a connected quasi-split reductive group, and its absolute root datum is $\fX$. In particular, $\cH_\fX$  is a reductive $\cF_\fX$-group scheme.
        \item \label{lem: scheme quasi-split fixed r.d.: 2}For every connected, quasi-split reductive group $\mathbf{G}$ over a field $F$, with absolute root datum $\fX$, there exists a point $x\in \cF_\fX(F)$ such that $\mathbf{G}\cong (\cH_\fX)_x$.
    \end{enumerate}
\end{lemma}
The proof closely follows the argument in \cite[Appendix B]{AA24}.
\begin{proof}
    By \cite[Theorem 6.1.16(2)]{Con14}, there exists a split reductive $\bZ$-group scheme  $(\cG_{\fX}, T_\fX, M_\fX)$ with root datum $\fX$. Choose a pinning $(\Phi^+, \{X_a\}_{a\in \Delta})$ for it (see Definition \ref{def: pinning}).

    Let $\Gamma$ be a finite group, and let $\alpha: \Gamma \to \Aut(\fX, \Delta)$ be a homomorphism. By \cite[Theorem 7.1.9(3)]{Con14}\footnote{The theorem is stated for the relative case, over some base scheme $S$. However by setting $S$ as the argument in the associated functors of points, we obtain the case in Definition \ref{def: group action on a scheme}.}, the pinning $(\Phi^+, \{X_a\}_{a\in \Delta})$ defines an isomorphism
    \begin{equation}\label{eq: rho: Aut(fx) to pinned auto}
        \rho: \Aut(\fX, \Delta) \overset{\sim}{\to} \Aut_\bZ(\cG_\fX, B_\fX, T_\fX,\{X_a\}_{a\in \Delta}) \subseteq \Aut_\bZ(\cG_{\fX}).
    \end{equation}
    Now, define a homomorphism $\alpha_2: \Gamma \to \Aut_\bZ(\cG_{\fX})$ by $\alpha_2 \coloneqq \rho \circ \alpha$.

    Let $\cF_\Gamma$ and $\cE_\Gamma$ be as in Lemma \ref{lem: fam of finite etale algebras}. Define the following internal morphism functor (see Definition \ref{def: internal hom})
    \[
        \underline{\cH'_{\fX,\Gamma}} \coloneqq \underline{( \cG_{\fX} \times_{\bZ} \cF_{\Gamma} ){}^{\wedge}_{\cF_{\Gamma}} \cE_{\Gamma}}.
    \]
    By Lemma \ref{lem: internal hom repr} it is representable by an $\cF_\Gamma$-affine scheme of finite type, which we denote as $\cH'_{\fX,\Gamma}$.

    Define $*_{\alpha_2}: \Gamma \to \Aut_{\cF_\Gamma}(\cH'_{\fX,\Gamma})$, by first acting by the action induced by the $\Gamma$-action on $\cE_\Gamma$ and then by the $\Gamma$-action induced from $\alpha_2$ on $\cG_\fX$. The notation is justified as this $\Gamma$-action resembles the construction in Definition \ref{def: cocycle to action}. Now, define the following functor (see Definition \ref{def: fixed points/quotient functor/space})
    \[
        \underline{\cH_{\fX,\alpha}} \coloneqq \underline{(\cH'_{\fX,\Gamma})^{*_{\alpha_2}(\Gamma)}}.
    \]
    By Lemma \ref{lem: fixed points scheme}, it is representable by an affine $\cF_\Gamma$-scheme.

    Denote $n_\fX:=C^{spt}(\dim (\fX))$, where $C^{spt}$ is the function given by \cite[Lemma B.1.1]{AA24}. Let \[
        D_\fX \coloneqq \left\{ (\Gamma,\kappa) \;\middle|\;
        \begin{aligned}
             & \Gamma\text{ is a group of index} \leq n_\fX ;                \\
             & \kappa \in \Mor(\Gamma,\Aut(\fX,\Delta))/ad(\Aut(\fX,\Delta))
        \end{aligned}
        \right\},
    \]
    and we claim that $D_\fX$ is finite. First, notice that there are finitely many finite group up to order $n_\fX$. Secondly, by \cite[Lemma B.3.2]{AA24}\footnote{We can use this lemma, although it requires $\Gamma$  to be abelian, as the abelian assumption is not used in the proof.}, there are also finitely many actions of a finite groups $\Gamma$ on a given root datum, up to adjunction.

    Now, for any equivalence class $\kappa \in \Mor(\Gamma,\Aut(\fX,\Delta))/ad(\Aut(\fX,\Delta))$ choose a representative $\alpha \in \kappa$. In addition, for any $(\Gamma,[\alpha]) \in D_\fX$ create a copy of $\cE_\Gamma\to\cF_\Gamma$ denotes by $(\cE_\Gamma)_{[\alpha]}\to(\cF_\Gamma)_{[\alpha]}$.  Finally, set
    $$
        \cF_\fX:=
        \bigsqcup_{(\Gamma,[\alpha]) \in D_\fX}
        (\cF_{\Gamma})_{[\alpha]},$$
    and
    $$
        \cH_{\fX}:=
        \bigsqcup_{(\Gamma,[\alpha]) \in D_\fX}
        \cH_{\fX,\alpha},
    $$where each $\cH_{\fX,\alpha}$ is build using the $(\cE_\Gamma)_{[\alpha]}\to(\cF_\Gamma)_{[\alpha]}$ copy.
    Define the morphism $\cH_{\fX} \to \cF_{\fX}$ as the disjoint union of the structure morphisms $\cH_{\fX,\alpha} \to (\cF_{\Gamma})_{[\alpha]}$.

    We claim that $\cH_{\fX} \to \cF_{\fX}$ satisfies the requirements of the lemma. For clause \eqref{lem: scheme quasi-split fixed r.d.: 1}, it suffices to consider the fibers at each $\cH_{\fX,\alpha} \to (\cF_{\Gamma})_{[\alpha]}$.

    By Lemma \ref{lem: fixed points scheme},  for any field $F$ and any point $x\in(\cF_{\Gamma})_{[\alpha]}(F)$:\[(\cH_{\fX,\alpha})_x = ((\cH'_{\fX,\Gamma})^{*_{\alpha_2}(\Gamma)})_x \cong ((\cH'_{\fX,\Gamma})_x)^{*_{\alpha_2}(\Gamma)}.
    \]
    Moreover, as $(\cE_\Gamma)_{[\alpha]}\to(\cF_\Gamma)_{[\alpha]}$ is finite and {\'e}tale we have $((\cE_\Gamma)_{[\alpha]})_x\cong\Spec(E)$ for some finite {\'e}tale $F$-algebra $E$ of degree $|\Gamma|$ \cite[Lemmas 29.36.4 and 29.36.7]{SP}. Therefore, by Lemma \ref{lem: internal hom base change},  we get that $(\cH'_{\fX,\Gamma})_x$, as an $F$-scheme represent the internal morphism functor $\underline{(\cG_\fX)_{E/F}}$. Notice, that $(\cH'_{\fX,\Gamma})_x$ can also be considered as an $\cF_\Gamma$-scheme by $x$.

    Denote by $(*_{\alpha_2})_{E/F}:\Gamma\to\Aut_F((\cG_\fX)_{E/F})$, the induced homomorphism from $*_{\alpha_2}$. We get that $(\cH_{\fX,\alpha})_x$ represents the functor $\underline{((\cG_\fX)_{E/F})^{(*_{\alpha_2})_{E/F}(\Gamma)}}$ (see Definition \ref{def: fixed points/quotient functor/space}). In conclusion, we get
    \begin{equation}\label{eq: fiber of H}
        (\cH_{\fX,\alpha})_x\cong ((\cG_\fX)_{E/F})^{(*_{\alpha_2})_{E/F}(\Gamma)}.
    \end{equation}

    Notice, the $\Gamma$-action of $\cE_\Gamma$ over $\cF_\Gamma$, induce a $\Gamma$-action on $\Aut_F((\cG_\fX)_{E/F})$. This $\Gamma$-action is trivial when reduced to the subgroup of $\Aut_F((\cG_\fX)_{E/F})$ consists of automorphism induced from $\Aut_F((\cG_\fX)_F)$. Therefore, $(\alpha_2)_{E/F}:\Gamma\to\Aut_F((\cG_\fX)_{E/F})$ is a 1-cocycle, and by Definition \ref{def: cocycle to action} $(*_{\alpha_2})_{E/F}=*_{((\alpha_2)_{E/F})}$. Therefore, clause \eqref{lem: scheme quasi-split fixed r.d.: 1} follows from Lemmas \ref{lem: fam of finite etale algebras}\eqref{lem: fam of finite etale algebras: 1}  and  \ref{lem: etale extension reduction - quasi-split}.

    Now, we show that clause \eqref{lem: scheme quasi-split fixed r.d.: 2} holds.  By \cite[Lemma B.1.1]{AA24}, there exists a finite Galois extension $E/F$ of degree $d \leq n_\fX$ such that $\mathbf{G}_E$ splits. Now, as $\mathbf{G}$ is quasi-split, it admits a Borel subgroup $\mathbf{B}\subseteq \mathbf{G}$, such that there exists a split maximal torus $\mathbf{T}\subseteq\mathbf{B}_E\subseteq\mathbf{G}_E$.   We can choose a pinning $\{X'_{a'} \}_{a' \in \Delta'}$ of $(\mathbf{G}_E,\mathbf{B}_E,\mathbf{T})$, in the sense of groups and not group schemes (recall Remark \ref{rem: pinning grp scheme / grp}).

    By \cite[Remark 6.1.2]{Con14}, there is a unique Borel subgroup scheme $T_\fX \subseteq B_{\fX}\subseteq \cG_{\fX}$ corresponding to $\Phi^+$, the positive system of roots we chose as part of the chosen pinning in the beginning of the proof. By Definition \ref{def: pinning}, we get that $((\cG_\fX)_E, (B_\fX)_E, (T_\fX)_E,\{(X_a)_E\}_{a\in \Delta})$ is a pinned reductive group.

    Now, by \cite[Proposition 1.5.5]{Con14}, as the based root datum are isomorphic, we get an $E$-isomorphism
    $$\alpha':((\cG_\fX)_E, (B_\fX)_E, (T_\fX)_E,\{(X_a)_E\}_{a\in \Delta})\to(\mathbf{G}_E,\mathbf{B}_E,\mathbf{T},\{X'_{a'} \}_{a' \in \Delta'}).$$
    Denote $\Gamma\coloneqq \Gal(E/F)$ and $c_{\alpha'}\in Z^1(\Gamma,\Aut_F((\cG_\fX)_E))$ as the 1-cocycle induced by $\alpha'$ (see Definition \ref{def: induced cocycle}).

    We show that $\Im(c_{\alpha'})\in\Aut_F((\cG_\fX)_E, (B_\fX)_E, (T_\fX)_E,\{(X_a)_E\}_{a\in \Delta})$. Notice, that $c_{\alpha'}$ preserve $ (B_\fX)_E$ as the $\Gamma$-action preserve $ (B_\fX)_E$ and $\mathbf{B}_E$. Therefore, $c_{\alpha'}$ also preserve $\mathbf{T}_E$ as $c_{\alpha'}((T_\fX)_E)$ is a maximal torus in $(B_\fX)_E$ and there is a unique one. $c_{\alpha'}$ also preserve $\{(X_\alpha)_E\}_{\alpha \in \Delta}$, and it can be checked directly by the equivalent characterization in \cite[Definition 1.5.4]{Con14}.

    By \cite[corollary 1.5.5]{Con14}, $\Aut_F((\cG_\fX)_E, (B_\fX)_E, (T_\fX)_E,\{(X_a)_E\}_{a\in \Delta})\cong \Out_F((\cG_\fX)_E)$. The cohomology class of $c_{\alpha'}$ in $H^1(\Gamma,\Out_F((\cG_\fX)_E))$ is an homomorphism, as the 1-coboundary condition, together with the quotient by inner automorphism, trivialize the $\Gamma$-action. Recall the homomorphism $\rho$ from \eqref{eq: rho: Aut(fx) to pinned auto}, and define  $\alpha''\coloneqq\rho^{-1}_E\circ c_{\alpha'}$, and let $\alpha$ the chosen representative of $[\alpha'']\in \Mor(\Gamma,\Aut(\fX,S))/ad(\Aut(\fX,S))$, and notice that it is cohomologous to $\alpha''$.

    Now, we construct a point $x\in\cF_\fX(F):=
        \bigsqcup_{(\Gamma,[\alpha]) \in D_\fX}
        (\cF_{\Gamma})_{[\alpha]}(F)$, as the point from  Lemma \ref{lem: fam of finite etale algebras}\eqref{lem: fam of finite etale algebras: 2} on the $[\alpha]$ indexed copy of $\cF_\Gamma$. We verify that $x$ satisfies clause \eqref{lem: scheme quasi-split fixed r.d.: 2}.

    By Equation \eqref{eq: fiber of H} and Lemma \ref{lem: cohomo cocycles iff iso forms} we are left to show that $((\cG_\fX)_{E/F})^{(*_{\alpha_2})_{E/F}(\Gamma)}$ and $\mathbf{G}$ induce cohomologous 1-cocycles in $Z^1(\Gamma,\Aut((\cG_\fX)_E))$. This follows as $((\cG_\fX)_{E/F})^{(*_{\alpha_2})_{E/F}(\Gamma)}$ satisfies the conditions of Lemma \ref{lem: etale extension reduction} for the 1-cocycle $\alpha$ which is cohomologous to $\alpha''$, and $\alpha''$  was induced from $\mathbf{G}$.
\end{proof}

Now, we can generalize Lemma \ref{lem: scheme quasi-split fixed r.d.}, to a versal family of quasi-split reductive groups up to a given rank. We do so by taking disjoint union of the schemes constructed in Lemma \ref{lem: scheme quasi-split fixed r.d.}, over all rood data up to the given rank. Notice, that if we want the scheme in our result to be of finite type, we need to ensure that the disjoint union is finite.

\begin{theorem} \label{thm: Scheme of quasi-split}
    For any integer  $n>0$, there is an affine scheme $\cF _{n}$ of finite type and a smooth $\cF _{n}$-affine group scheme $\cH_n \to \cF_n$ of finite type, such that:
    \begin{enumerate}
        \item \label{thm: scheme quasi-split: 1}For every field $F$ and every $x\in\cF_n(F)$, the group $(\cH_n)_x$ is a connected quasi-split reductive group of rank at most $n$. In particular, $\cH_n$  is a reductive group scheme.
        \item \label{thm: scheme quasi-split: 2}For every connected, quasi-split reductive group $\mathbf{G}$ of rank at most $n$ over a field $F$ there exists a point $x\in \cF_n(F)$ such that $\mathbf{G}\cong (\cH_n)_x$.
    \end{enumerate}
\end{theorem}

\begin{proof}
    By \cite[Lemma B.3.1]{AA24} there are finitely many root data of a given rank. Let $\cH_\fX\to \cF_\fX$ be as in Lemma \ref{lem: scheme quasi-split fixed r.d.}. Set
    \[
        \cF_n \coloneqq
        \bigsqcup_{\text{root datum } \fX \text{ of rank } \leq\,n }
        \cF_{\fX},
    \]
    and
    \[
        \cH_n \coloneqq
        \bigsqcup_{\text{root datum } \fX \text{ of rank } \leq\,n }
        \cH_{\fX}.
    \]
    Now, clause \eqref{thm: scheme quasi-split: 1} follows as it is a local property, and clause \eqref{thm: scheme quasi-split: 2} follows as the union runs over all required root data.

\end{proof}

\subsection{Versal family of reductive groups}
In the previous section, we constructed a reductive group scheme encoding all the quasi-split reductive groups up to a given rank. In this section we add to this construction all of the inner forms of these quasi-split reductive groups. Therefore, we acquire a versal family of reductive groups, as any reductive group has a (unique) quasi-split inner form (\cite[19.57]{Milne_reductive_groups}).

The first step is to prove that the 1-cocycle functor, which we now define, is representable.

\begin{definition}\label{def: 1-cocycle func}
    Let $\Gamma$ be a finite group, $S$ a scheme, and $\mathbf{G}$ an $S$-group scheme with a $\Gamma$-action. Define the contra-variant functor $\underline{Z^1_{\Gamma,\mathbf{G}}}:\mathbf{Sch}_S\to \mathbf{Set}$, called the \textbf{1-cocycle functor}, by
    \[
        \underline{Z^1_{\Gamma,\mathbf{G}}}(S') = \left\{ \phi : \Gamma \to \mathbf{G}(S') \ \middle|\ \forall \sigma, \tau \in \Gamma, \ \phi(\sigma\tau) = \phi(\sigma) \cdot \prescript{\sigma}{}{\phi(\tau)} \right\},
    \]
    for any $S$-scheme $S'$, where $\prescript{\sigma}{}{(-)}$ denotes the action of $\sigma\in\Gamma$ on $G(S')$.
\end{definition}

\begin{lemma} \label{lem: 1-cocycles scheme}
    Let $\Gamma$, $S$ and $\mathbf{G}$ be as in Definition \ref{def: 1-cocycle func}, then there is an $S$-scheme $Z^1_{\Gamma,\mathbf{G}}$ representing $\underline{Z^1_{\Gamma,\mathbf{G}}}$. Moreover, if $\mathbf{G}$ is affine, then so does $Z^1_{\Gamma,\mathbf{G}}$. Additionally, there are morphisms
    $$\{c_\gamma:Z^1_{\Gamma,\mathbf{G}}\to \mathbf{G}\}_{\gamma\in \Gamma},$$
    such that for any $S$-scheme $S'$, $S$-morphism $\phi': S'\to Z^{1}_{\Gamma,\mathbf{G}}$, and $\phi\in \underline{Z^1_{\Gamma,\mathbf{G}}}(S')$ induced from $\phi'$ by representability, we get
    \[
        c_\gamma \circ \phi' = \phi(\gamma).
    \]
\end{lemma}
\begin{proof}
    We construct the scheme $Z^1_{\Gamma,\mathbf{G}}$ as the limit of a finite diagram of $S$-schemes describing the 1-cocycle condition. Lemma \ref{lem: finite lim of scheme exist} ensures that such a limit exists and affine, as long as all the objects in the diagram are affine.

    Consider the $S$-scheme $\mathbf{G}^{|\Gamma|}\coloneqq\prod_{\gamma\in\Gamma}\mathbf{G}_\gamma$, where $\mathbf{G}_\gamma$ is a copy of $\mathbf{G}$ for each $\gamma \in \Gamma$. For any pair $\sigma, \tau \in \Gamma$, such that $\sigma\tau=\gamma$, we define a map $c'_{\sigma,\tau}:\mathbf{G}^{|\Gamma|}\to \mathbf{G}_{\gamma}$, encoding the 1-cocycle condition, as the composition of the following morphisms:
    \[
        c'_{\sigma,\tau}: \mathbf{G}^{|\Gamma|} \xrightarrow{p_\sigma\times p_\tau}
        \mathbf{G}_{\sigma} \times \mathbf{G}_{\tau} \xrightarrow{id \times \prescript{\sigma}{}{(-)}}
        \mathbf{G}_{\sigma} \times \mathbf{G}_{\tau}
        \xrightarrow{\cong^{\sigma}_{\gamma} \times \cong^{\tau}_{\gamma}}
        \mathbf{G}_{\gamma} \times \mathbf{G}_{\gamma}
        \xrightarrow{m_\gamma}
        \mathbf{G}_{\gamma}
    \]
    where:
    \begin{enumerate}
        \item $p_\sigma$ is the projection to the $\sigma$ coordinate, for any $\sigma\in\Gamma$.
        \item $\prescript{\sigma}{}{(-)}$ is the action of $\sigma$ by the $\Gamma$-action on $\mathbf{G}$.
        \item $\cong^\sigma_\gamma$ is the isomorphism between $\mathbf{G}_{\sigma}$ to $\mathbf{G}_{\gamma}$, for any $\sigma,\gamma\in\Gamma$.
        \item $m_\gamma$ is the multiplication of $\mathbf{G}_{\gamma}$ as a group scheme, for any $\gamma\in\Gamma$.
    \end{enumerate}

    Now, we let $(Z^1_{\Gamma,\mathbf{G}},\{c''_{\gamma}:Z^1_{\Gamma,\mathbf{G}}\to \mathbf{G}_{\gamma}\}_{\gamma \in \Gamma},\iota: Z^1_{\Gamma,\mathbf{G}}\to \mathbf{G}^{|\Gamma|})$ be the limit of the following diagram:
    
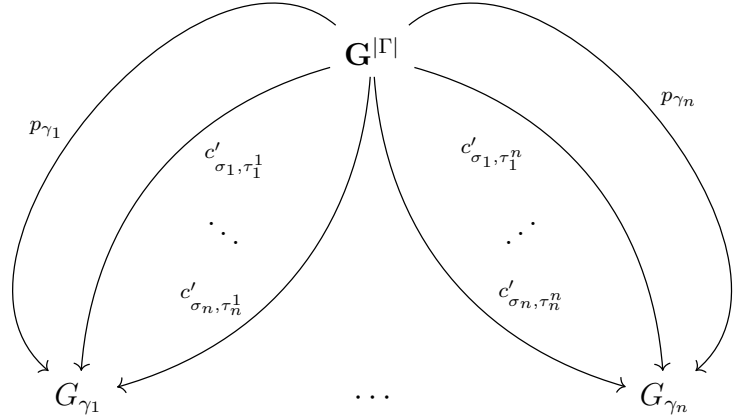
\[
        \begin{tikzcd}[row sep=1.5cm, column sep=1cm]
            && \mathbf{G}^{|\Gamma|}
            \arrow[lldd, bend right = 90, "p_{\gamma_1}"{xshift=-5ex}]
            \arrow[lldd, bend right = 35, pos=0.35, "c'_{\sigma_1,\tau^1_1}"]
            \arrow[lldd,bend left = 35, "c'_{\sigma_n,\tau^1_n}"{xshift=-9ex}]
            \arrow[rrdd, bend left = 90, "p_{\gamma_n}"]
            \arrow[rrdd, bend left = 35, "c'_{\sigma_1,\tau^n_1}"{xshift=-10ex}]
            \arrow[rrdd,bend right = 35, pos=0.65,"c'_{\sigma_n,\tau^n_n}"]
            && \\
            & \ddots &&  \iddots  & \\
            G_{\gamma_1} && \cdots && G_{\gamma_n}
        \end{tikzcd}
    \]
    where  $\sigma_i$ vary through $\Gamma$ and $\tau^i_j\coloneqq\sigma^{-1}_j\gamma_i$.

    We set $c_\gamma \coloneqq \cong^\gamma\circ c''_{\gamma}$, where $\cong^\gamma:\mathbf{G}_\gamma\to \mathbf{G}$ is an isomorphism, and verify that this construction satisfy the requirement of the lemma. For any $\phi':S'\to Z^1_{\Gamma,\mathbf{G}}$, set $\phi(\gamma) \coloneqq c_\gamma\circ\phi' \in \mathbf{G}(S')$ for any $\gamma\in\Gamma$. Notice, that our construction ensures, that for any $\sigma,\tau\in\Gamma$, we have $c_{\sigma\tau}=c_\sigma\cdot\prescript{\sigma}{}{c_\tau}$, where this notation should be clear from context. Therefore, $\phi$ also satisfies the cocycle condition.   \end{proof}
Similar to the proof of theorem \ref{thm: Scheme of quasi-split}, we first construct a family for a fixed root datum. We then generalize the result to all root data up to a given rank, by taking a (finite) disjoint union.

\begin{lemma} \label{lem: versal family reductive fixed r.d.}
    Let $\fX$ be a root datum. There is a scheme $\cS_\fX$  of finite type and a smooth $\cS_\fX$-affine group scheme of finite type $\cR_\fX \to\cS_\fX$ such that:
    \begin{enumerate}
        \item \label{lem: versal family reductive fixed r.d.: 1}For every field $F$ and every $x\in\cS_\fX(F)$, the group $(\cR_\fX)_x$ is connected, reductive and its absolute root datum is $\fX$. In particular, $\cR_\fX$  is a reductive $\cS_\fX$-group scheme.
        \item \label{lem: versal family reductive fixed r.d.: 2}For every connected, reductive group $\mathbf{G}$ over a field $F$, with absolute root datum $\fX$, there exists a point $x\in \cS_\fX(F)$ such that $\mathbf{G}\cong (\cR_\fX)_x$.
    \end{enumerate}
\end{lemma}
Notice, that by Lemma \ref{lem: scheme quasi-split fixed r.d.}, we already have a versal family of quasi-split reductive groups with root datum $\fX$. Therefore, we are left to add all inner forms to our  construction (see e.g. \cite[19.57]{Milne_reductive_groups}).
\begin{proof}
    Let $\Gamma$ be a finite group, $\cH_\fX\to\cF_\fX$ be as in Lemma \ref{lem: scheme quasi-split fixed r.d.}, and set $\underline{\cR'_{\fX,\Gamma}}\coloneqq \underline{\cH_{\fX}\, ^{\,\wedge}_{\cF_{\Gamma}} \cE_\Gamma}$. Notice, that by Lemma \ref{lem: internal hom repr} it is representable by an affine scheme of finite type, which we denote by $\cR'_{\fX,\Gamma}$. Moreover, $\cR'_{\fX,\Gamma}$ is a connected reductive $\cF_\Gamma$-group scheme as for any algebraically closed field $k$ and a geometric point $\overline{x}\in\cF_{\Gamma}(k)$
    \[
        (\cR'_{\fX,\Gamma})_{\overline{x}}\cong(\cH_{\fX})_{k^{|\Gamma|}/k} \cong \prod_{i=1}^{|\Gamma|}(\cH_{\fX})_k.
    \]

    Therefore, by \cite[Theorem 7.1.9(1)]{Con14}, the functor $\underline{\Aut_{\cF_\Gamma}(\cR'_{\fX,\Gamma})}:\mathbf{Sch}_{\cF_\Gamma}\to\mathbf{Set}$ is representable by an $\cF_\Gamma$-scheme denoted by $\Aut_{\cF_\Gamma}(\cR'_{\fX,\Gamma})$. Moreover, the sub-functor $\underline{\Inn(\cR'_{\fX,\Gamma})}$  is representable by a closed $\cF_\Gamma$-subscheme of $\Aut_{\cF_\Gamma}(\cR'_{\fX,\Gamma})$, denoted by $\Inn(\cR'_{\fX,\Gamma})$. In addition, the scheme $\Inn(\cR'_{\fX,\Gamma})$ inherits the $\Gamma$-action from $\Aut(\cR'_{\fX,\Gamma})$, as can be directly checked on the functor of points.

    Set $\cS_{\fX,\Gamma}$ as the $\cF_\Gamma$-scheme,  $Z^1_{\Gamma,\Inn(\cR'_\fX)}$ from Lemma \ref{lem: 1-cocycles scheme} representing the 1-cocycles from $\Gamma$ to $\Inn(\cR'_{\fX,\Gamma})$, and set $\cR''_{\fX,\Gamma}\coloneqq\cR'_{\fX,\Gamma}\times_{\cF_{\Gamma}}    \cS_{\fX,\Gamma}$.

    We construct for each $\gamma \in \Gamma$ a morphism $*_\gamma:\cR''_{\fX,\Gamma} \to \cR'_{\fX,\Gamma}$, that correspond to Definition \ref{def: cocycle to action}, by: \[
        *_\gamma = \mu \circ (\gamma \times c_{\gamma}),
    \]
    where:
    \begin{enumerate}
        \item $c_\gamma:\cS_{\fX,\Gamma}\to\Inn(\cR'_{\fX,\Gamma})$ is the morphism from Lemma \ref{lem: 1-cocycles scheme};
        \item $\gamma :\cR'_{\fX,\Gamma}\to \cR'_{\fX,\Gamma}$ is the action of $\gamma$ by the $\Gamma$-action on $\cR'_{\fX,\Gamma}$;
        \item $\mu:\cR'_{\fX,\Gamma}\times\Inn(\cR'_{\fX,\Gamma})\to \cR'_{\fX,\Gamma}$ is the $\Inn(\cR'_{\fX,\Gamma})$-action on $\cR'_{\fX,\Gamma}$.
    \end{enumerate}

    Now, let $\cR_{\fX,\Gamma}$ be the limit of the diagram consists of these $\{*_\gamma\}_{\gamma  \in\Gamma}$. Notice, that $\cR_{\fX,\Gamma}$ inherits a $\cS_{\fX,\Gamma}$-structure from $\cR''_{\fX,\Gamma}$.

    Finally, with resemblance to the proof of Lemma \ref{lem: scheme quasi-split fixed r.d.}, denote $n_\fX:=C^{spt}(\dim (\fX))$, where $C^{spt}$ is the function given by \cite[Lemma B.1.1]{AA24}. Define
    \[
        \cS_\fX=\bigsqcup_{\Gamma \text{ group of index } \leq n_\fX} \cS_{\fX,\Gamma}
    \]
    and
    \[
        \cR_\fX=\bigsqcup_{\Gamma \text{ group of index } \leq n_\fX} \cR_{\fX,\Gamma}.
    \]We show that $\cR_\fX \to \cS_\fX$ satisfies the requirement of the lemma.

    For clause \eqref{lem: versal family reductive fixed r.d.: 1}, we may consider each $\cR_{\fX,\Gamma} \to \cS_{\fX,\Gamma}$ separately.  Let $x$ be a point in $\cS_{\fX,\Gamma}(F)$, and notice that we also get a unique point $s\in\cF_\Gamma(F)$ by composing with the structure morphism $\cS_{\fX,\Gamma}\to\cF_{\Gamma}$.

    Let $E$ be an {\'e}tale $F$-algebra such that $\Spec(E)\cong(\cE_\Gamma)_s$. Notice, that
    \[(\cR'_{\fX,\Gamma})_x\cong((\cH_{\fX,\Gamma})_{s})_{E/F},
    \]
    Moreover, $x$ correspond to a unique 1-cocycle $c_x$ such that
    $$c_x\in Z^1(\Gamma,\Inn(\cR'_{\fX,\Gamma})(F))=Z^1(\Gamma,\Inn(((\cH_{\fX,\Gamma})_s)_{E/F})(F))\hookrightarrow Z^1(\Gamma,\Aut(((\cH_{\fX,\Gamma})_s)_{E/F})(F))).$$
    In addition,
    \[
        (\cR''_{\fX,\Gamma})_x \cong
        (\cR'_{\fX,\Gamma})_x\times_{(\cF_{\Gamma})_s}(\cS_{\fX,\Gamma})_x\cong
        ((\cH_{\fX,\Gamma})_s)_{E/F}
    \]
    over $F$.

    For any $F$-scheme $T$, and a point in $(\cR''_{\fX,\Gamma})_x(T)$, our construction induce a unique 1-cocycle in $Z^1(\Gamma,\Aut(\cH_{\fX,\Gamma})_{E/F}(T)))$ as the base change of $c_x$ by the structure morphism $T\to \Spec(F)$. Therefore, for any $\gamma \in \Gamma$, $(*_\gamma)_x=*_{c_x}(\gamma)$ (recall Definition \ref{def: cocycle to action}), and$$ (\cR_{\fX,\Gamma})_x\cong  (((\cH_{\fX,\Gamma})_s)_{E/F})^{*_{c_x}(\Gamma)}.
    $$
    Finally, clause \eqref{lem: versal family reductive fixed r.d.: 1} follows from Lemmas \ref{lem: fam of finite etale algebras}\eqref{lem: fam of finite etale algebras: 1}  and  \ref{lem: etale extension reduction}.

    Now, for clause \eqref{lem: versal family reductive fixed r.d.: 2}, notice that by \cite[19.57]{Milne_reductive_groups} $\mathbf{G}$ has a unique quasi-split inner $F$-form, which we denote by $\mathbf{G}^{qs}$. Therefore, $\mathbf{G}_E\cong\mathbf{G}^{qs}_E$ for some finite Galois extension $E/F$.  Denote $\Gamma\coloneqq\Gal(E/F)$, and notice that as $G^{qs}$ is an inner $F$-form, the image of the induced 1-cocycle $c'\in Z^1(\Gamma,\Aut(\mathbf{G}_E))$ lies in $\Inn(\mathbf{G}_E)$.

    Now, by Lemma \ref{lem: scheme quasi-split fixed r.d.} we get a point $s\in\cF_\Gamma(F)$ such that $(\cH_\fX)_s\cong\mathbf{G}^{qs}$ as the root datum of $\mathbf{G}^{qs}$ is isomorphic to the root datum of $\mathbf{G}$. Therefore, we get an isomorphism $((\cH_\fX)_s)_E\cong \mathbf{G}_E^{qs}\cong \mathbf{G}_E$, that induces a 1-cocycle  $c\in Z^1(\Gamma,\Aut(((\cH_\fX)_s)_E)$ (see Remark \ref{def: induced cocycle}).

    One can check on the associated functor of point, that the image of $c$ is in $\Inn(((\cH_\fX)_s)_E)$. Therefore, $c$ induce a point $x\in\cS_{\fX,\Gamma}(F)$. The verification that $x$ satisfies clause \eqref{lem: versal family reductive fixed r.d.: 2}, namely that $\mathbf{G}\cong(\cR_\fX)_x$, is similar to Lemma \ref{lem: scheme quasi-split fixed r.d.} and uses the proof of clause \eqref{lem: versal family reductive fixed r.d.: 1} above.
\end{proof}
We have constructed everything needed to prove the main result of this paper.

\begin{theorem}\label{thm: versal family reductive}
    For any integer $n>0$, there exists a scheme $\cS _n$ of finite type and a smooth  $\cS_n$-group scheme of finite type $\cR_n\to \cS_n$ such that:
    \begin{enumerate}
        \item \label{thm: versal family reductive: 1}For any field $F$ and every $x\in\cS_n(F)$, the group $(\cR_n)|_x$ is connected and reductive of rank at most $n$. In particular, $\cR_n$  is a reductive $\cS_n$-group scheme.
        \item \label{thm: versal family reductive: 2}For any connected and reductive group $\mathbf{G}$, of rank at most $n$ over a field $F$,  there exists a point  $x\in \cS_n(F)$, such that $\mathbf{G}\cong (\cR_n)|_x$.
    \end{enumerate}
\end{theorem}
The method of the proof strongly resembles the method used to generalize Lemma \ref{lem: scheme quasi-split fixed r.d.} to Theorem \ref{thm: Scheme of quasi-split}.

\begin{proof}[Proof of Theorem \ref{thm: versal family reductive}]
    By \cite[Lemma B.3.1]{AA24} there are finitely many root data up to a given rank. Let $\cR_\fX\to \cS_\fX$ be as in Lemma \ref{lem: versal family reductive fixed r.d.}. Set
    \[
        \cS_n \coloneqq
        \bigsqcup_{\text{root datum } \fX \text{ of rank } \leq\,n }
        \cS_{\fX},
    \]
    and
    \[
        \cR_n \coloneqq
        \bigsqcup_{\text{root datum } \fX \text{ of rank } \leq\,n }
        \cR_{\fX}.
    \]
    Now, clause \eqref{thm: versal family reductive: 1} follows as it is a local property, and clause \eqref{thm: versal family reductive: 2} follows as the union runs over all required root data.
\end{proof}
\begin{remark}
    Notice, that we could have bypassed the quasi-split construction of Lemma \ref{lem: scheme quasi-split fixed r.d.} and Theorem \ref{thm: Scheme of quasi-split}, as the functor $\underline{\Aut(\cG_{\fX})}$ is itself representable, and we could have constructed the scheme of Theorem \ref{thm: versal family reductive} directly from the representing scheme of $\underline{\Aut(\cG_{\fX})}$. We chose to split our proof to give a more modular construction, and to establish Theorem \ref{thm: Scheme of quasi-split}, as it has a significance of its own.
\end{remark}

\bibliographystyle{alpha}
\bibliography{VersalFamilyOfReductiveGroups.bib}

\end{document}